\documentclass[11pt]{amsart}
\usepackage{amsfonts}
\usepackage{amsfonts,latexsym,rawfonts,amsmath,amssymb,amsthm}
\usepackage[plainpages=false]{hyperref}

\usepackage{graphicx}

\numberwithin{equation}{section}

\newcommand{\beq}{\begin{equation}}
\newcommand{\eeq}{\end{equation}}
\newcommand{\beqs}{\begin{eqnarray*}}
\newcommand{\eeqs}{\end{eqnarray*}}
\newcommand{\beqn}{\begin{eqnarray}}
\newcommand{\eeqn}{\end{eqnarray}}
\newcommand{\beqa}{\begin{array}}
\newcommand{\eeqa}{\end{array}}

\def\lra{\longrightarrow}

\def\p{\prime}

\def\bc{\begin{center}}
\def\ec{\end{center}}

\def\p{\partial}

\def\begeq{\begin{equation}}
\def\endeq{\end{equation}}
\def\and{\quad{\rm and}\quad}

\let\lra=\longrightarrow

\def\mapright\#1{\,\smash{\mathop{\lra}\limits^{\#1}}\,}

\newtheorem{prop}{Proposition}[section]
\newtheorem{theo}[prop]{Theorem}
\newtheorem{lem}[prop]{Lemma}

\newtheorem{cor}[prop]{Corollary}
\newtheorem{rem}[prop]{Remark}

\title  {Stability on  K\"ahler-Ricci flow,  I}

\author { Xiaohua Zhu}


\address{ Xiaohua Zhu\\Department of Mathematics, Peking University,
Beijing, 100871, China\\
 xhzhu@math.pku.edu.cn}
\thanks {  Partially supported by  NSF10425102 in China.}

\begin{document}
\bibliographystyle{plain}

\begin{abstract} In this paper, we prove that K\"ahler-Ricci flow
converges to a K\"ahler-Einstein metric (or a K\"ahler-Ricci
soliton) in the sense of Cheeger-Gromov as long as an initial
K\"ahler metric is very closed to $g_{KE}$ (or $g_{KS}$) if a
compact K\"ahler manifold with $c_1(M)>0$  admits a K\"ahler
Einstein metric $g_{KE}$  (or a K\"ahler-Ricci soliton $g_{KS}$).
The result  improves Main Theorem in [TZ3] in the sense of
stability of K\"ahler-Ricci flow.
\end{abstract}
\maketitle

\setcounter{section}{-1}


\section{Introduction}

The Ricci flow was first introduced by R. Hamilton in [Ha]. If the
underlying manifold $M$ is  K\"ahler with positive first Chern
class $c_1(M)>0$, it is more natural to study the following
K\"ahler-Ricci flow (normalized),
  \begin{align} &\frac {\partial g(t,\cdot)}{\partial t}
=-\text{Ric}(g(t,\cdot))+g(t,\cdot),\notag\\
&g(0,\cdot)=g,
\end{align}
 where $g$ is an initial K\"ahler metric with its K\"ahler form $\omega_g \in 2\pi
 c_1(M)>0.$
 It can be shown that (0.1) preserves the K\"ahler class.
 Moreover,  (0.1) has a global solution $g_t=g(t,\cdot)$  for any $t>0$ ([Ca]). So,
 the main interest and difficulty of (0.1) is to study the limiting  behavior of $g_t$
  as $t$ tends to $\infty$ (cf. [CT1], [CT2], [TZ3],  etc.).

In this paper, we study a stability problem of K\"ahler-Ricci flow
(0.1), namely, we assume that  $M$ admits a K\"ahler-Einstein
metric or a K\"ahler-Ricci soliton, and then we analysis the
behavior of evolved K\"ahler metrics $g_t$ of (0.1). We shall
prove

\begin{theo}[Main Theorem] Let $M$ be  a compact K\"ahler manifold with  $c_1(M)>0$ which admits a
K\"ahler Einstein metric $g_{KE}$ (or a K\"ahler-Ricci soliton
$(g_{KS,X_0})$ with respect some holomorphic vector field $X_0$ on
$M$) with its K\"ahler form in $2\pi c_1(M)$. Let $\psi$ be a
K\"ahler potential of an initial metric $g$ of (0.1) and
$\varphi=\varphi_t$ be a family of K\"ahler potentials of evolved
metrics $g_t$ of (0.1), i.e.,
$\omega_g=\omega_{KE}+\sqrt{-1}\p\bar\p\psi$ (or
$\omega_g=\omega_{KS}+\sqrt{-1}\p\bar\p\psi$) and
$\omega_\varphi=\omega_{KE}+\sqrt{-1}\p\bar\p\varphi$ (or
$\omega_g=\omega_{KS}+\sqrt{-1}\p\bar\p\varphi$), where
$\omega_g$, $\omega_\varphi=\omega_{g_t}$ and $\omega_{KE}$ (or
$\omega_{KS}$) denote K\"ahler forms of $g$, $g_t$ and $g_{KE}$
(or $g_{KS}$), respectively. Then there exists a small $\epsilon$
such that if
$$\|\psi-\underline\psi\|_{C^{2,\alpha}}\le \epsilon,$$
 where $\underline\psi=\frac{1}{\int_M\omega_{KE}^n}
 \int_M\psi\omega_{KE}^n$ (or $\underline\psi=\frac{1}{\int_M\omega_{KS}^n}
 \int_M\psi\omega_{KS}^n$), then  there
exist a family of  holomorphisms  $\sigma=\sigma_t$ on $M$ such
that K\"ahler potentials
$(\varphi_\sigma-\underline{\varphi_\sigma})$ are $C^k$-norm
uniformly bounded, where $\varphi_\sigma=\sigma^*\varphi+\rho$ and
$\rho=\rho_t$ are K\"ahler potentials defined by
 $\rho^*(\omega_{KE})=\omega_{KE}+\sqrt{-1}\partial\overline\partial\rho$
 and $
\int_M e^{-\rho}\omega_{KE}^n=\int_M\omega_{KE}^n$
 (or
 $\rho^*(\omega_{KS})=\omega_{KS}+\sqrt{-1}\partial\overline\partial\rho$
 and $\int_M e^{-\rho-X_0(\rho)}\omega_{KS}^n=\int_M\omega_{KS}^n$).
As a consequence, $g_t$ converge to $g_{KE}$ (or $g_{KS}$)
smoothly in the sense of Cheeger-Gromov.
\end{theo}

The main step in the proof of Theorem 0.1 is to obtain a decay
estimate for $\dot\varphi$ and $\varphi$ both when one studies the
convergence of K\"ahler-Ricci flow as in [CT2], [PS], [TZ3] etc.
In case that  $M$ admits a K\"ahler Einstein metric or $M$ admits
a K\"ahler-Ricci soliton and an initial potential $\psi$ is
$K_{X_0}$-invariant, we can obtain an exponential decay estimate
for both $\dot\varphi$ and $\varphi$, so  we can improve that
K\"ahler potentials $(\sigma^*\varphi+\rho)$ in the theorem
exponentially converge to zero as long as
$\|\psi-\underline\psi\|_{C^{2,\alpha}}$ is small, where $K_{X_0}$
is an one-parameter compact subgroup generated by the imaginary
part $X'$ of  $X_0$ ([TZ1],[TZ2]). This result is also obtained in
[TZ3] where a crucial step is to use the  monotonicity and  the
properness of the Mabuchi's K-energy on a K\"ahler-Einstein
manifold with $c_1(M)>0$ (or the monotonicity and  the properness
of the generalized K-energy on a compact K\"ahler manifold which
admits a K\"ahler-Ricci soliton , cf. [CTZ]). But at the present
paper, we avoid to use these energies in our case of the stability
problem. This advantage allows us to remove the
$K_{X_0}$-invariant condition for the initial potential $\psi$  in
case of K\"ahler-Ricci soliton  in  Theorem 0.1, although we need
more careful computations than the case of K\"ahler-Einstein
metric. Basically, we shall use the generalized Futaki-invariant
and the Gauge Transformation induced by the reductive subgroup
$\text{Aut}_r(M)$  of holomorphisms transformation group
$\text{Aut}(M)$ on $M$ to control the modified K\"ahler potentials
$(\sigma^*\varphi+\rho)$ along the K\"ahler-Ricci flow. We note
that the definition of generalized Futaki-invariant is independent
of the choice of K\"ahler metric, which needs no
$K_{X_0}$-invariant condition ([TZ2]).  Unfortunately, we could
not improve the convergence of $(\sigma^*\varphi+\rho)$
exponentially without the assumption of $K_{X_0}$-invariant
condition. But we believe that it is still true if one can extend
the Gauge Transformation $\text{Aut}_r(M)$ to $\text{Aut}(M)$ (cf.
Proposition 2.10).

 Theorem 0.1 will be proved in Section 1 and Section 2 while in
Section 1 we consider the case of K\"ahler-Einstein metric and in
Section 2, we consider the case of K\"ahler-Ricci soliton. The
rest of paper is as follows: In Section 3, we prove  a uniqueness
result for the limit of K\"ahler-Ricci flow as an application of
Theorem 0.1; Section 4 and Section 5 are two appendixes, one is a
lemma about a $W^{k,2}$-estimate for $\dot\varphi_t$ and another
is a lemma about the existence of almost orthonormality  of a
K\"ahler potential to the space of first eigenvalue-functions of
operator $(P,\omega_{KS})$ defined in Lemma 2.2 in Section 2.

The author would like to thank professor Gang Tian and professor
Xiuxiong Chen for valuable discussions.

\section{In case of K\"ahler-Einstein metric}

In this section, we assume that $M$ admits a K\"ahler Einstein
metric $g_{KE}$ with its K\"ahler form $\omega_{KE}\in 2\pi
c_1(M)$. For simplicity, we set a class of K\"ahler potentials by
\begin{align}\mathcal{M}(\omega_{KE})=\{\phi\in
C^\infty(M,R)|~~\omega_{KE}+\sqrt{-1}\p\bar\p\phi>0\}\notag.
 \end{align}
Then for any K\"ahler metric $g$ with its K\"ahler form
$\omega_g\in 2\pi c_1(M)$, we have
$\omega_g=\omega_{KE}+\sqrt{-1}\p\bar\p\psi$ for some $\psi\in
\mathcal{M}(\omega_{KE})$ and K\"ahler-Ricci flow (0.1) is
equivalent to a parabolic equation of  complex Monge-Amp\`ere type
for K\"ahler potentials $\varphi_t=\varphi(t,\cdot)$ with
$\omega_{g_t}=\omega_{KE}+\sqrt{-1}\p\bar\p\varphi_t$,
\begin{align}
&\frac{\p\varphi}{\p
t}=\log\frac{\omega^n_\varphi}{\omega_{KE}^n}+\varphi, \notag\\
&\varphi(0)=\psi-\underline{\psi},
\end{align}
 where $\underline\psi=\frac{1}{V}\int_M\psi\omega_{KE}^n$ and $V=\int_M\omega_{KE}^n$.

Set a H\"older space by
$$\mathcal{K}(\epsilon_0)=\{\phi\in\mathcal{M}(\omega_{KE})|~~\|\phi-\underline\phi\|_{C^{2,\alpha}}\leq\epsilon_0\}.
$$
  Let $\text{Aut}_0(M)$ be the
connected component of holomorphisms transformation group of $M$
which contains the identity map of $M$.  Then we shall prove

\begin{theo} There exists a small $\epsilon$ such that for any initial data $\psi\in
\mathcal{K}(\epsilon)$ in equation (1.1),
$\|\varphi-\underline{\varphi}\|_{C^{2,\alpha}}$  are uniformly
bounded, where $\varphi=\varphi_t=\varphi(t,\cdot)$ are evolved
K\"ahler potentials of (1.1).  Moreover,  there exist a family of
$\sigma=\sigma_t\in\text{Aut}_0(M)$  such that K\"ahler potentials
$(\varphi_\sigma-\underline{\varphi_\sigma})$ converge
exponentially to $0$ as $t\to \infty$, where
$\varphi_\sigma=(\sigma^*\varphi+\rho)$, and $\rho=\rho_t$ are
K\"ahler potentials defined by
\begin{align}
&\sigma^*(\omega_{KE})=\omega_{KE}+\sqrt{-1}\partial\overline\partial\rho,\notag\\
&\int_M e^{-\rho}\omega_{KE}^n=\int_M\omega_{KE}^n.
 \end{align}
As a consequence, K\"ahler metrics $\sigma^*(\omega_{\varphi})$
converge exponentially to $\omega_{KE}$.
\end{theo}

We need several lemmas to prove Theorem 1.1.  Let
$\Lambda_1(M,\omega_{KE})$ be a finite dimensional linear space of
the first eigenvalue-functions of Lapalace
 operator $\triangle_{\omega_{KE}}$ associated to the metric $\omega_{KE}$. Then by using the
 Bochner formula, it is well-known that the first non-zero eigenvalue
 is $1$ and
 $\Lambda_1(M,\omega_{KE})=\text{span}\{\theta_X|~X\in\eta(M)\}$,
 where $\eta(M)$ is a linear space consisting of holomorphic vector fields on
 $M$ which is isomorphic  to the Lie algebra of $\text{Aut}_0(M)$
 and $\theta_X$ is a potential of $X$  defined by
\begin{align}&\sqrt{-1}\overline\partial\theta_X=i_X(\omega_{KE}),\notag\\
&\int_M\theta_X\omega_{KE}^n=0.
\end{align}

 By using the continuity of   eigenvalues  of Lapalcian operators, one sees

\begin{lem} Let $\lambda_1(\omega_\phi)$ and  $\lambda_2(\omega_\phi)$ be the first and the second
eigenvalues of Lapalcian operator associated to K\"ahler metric
$\omega_\phi$, respectively. Then there exists a $\delta_0$ such
that for any $\phi\in\mathcal{K}(\epsilon_0)$, we have
 \beqs &&\lambda_1(\omega_\phi)\ge 1+\delta_0,~~\text{if}~~ \eta(M)= 0,\\
&&\lambda_2(\omega_\phi)\ge 1+\delta_0, ~~ \text{if}~~ \eta(M)\ne
0,\eeqs
 where $\epsilon_0$ is a small positive number.
 \end{lem}

  Fix a  large number $T$ and $N$, we can choose a sufficient small $\epsilon$
depends on $T$, $\epsilon_0$ and $N$ such that for any  $t\leq T$,
 evolved K\"ahler potentials $\varphi_t$ of (1.1) lie in  $\mathcal{K}(\frac{\epsilon_0}{2})$
 and satisfy
\begin{align}
|\dot\varphi_t-c(t)|_{C^0}\leq (\frac{\epsilon_0}{2N})^2, \text{
and } osc(\varphi_t)\leq\frac{\epsilon_0}{4N},
\end{align}
whenever $\|\psi-\underline\psi\|_{C^{2,\alpha}}\leq\epsilon$.
Here $c(t)=\frac{1}{V}\int_M\dot\varphi_t\omega_{\varphi_t}^n.$
 Choose a maximal $\delta(T)$ such that $\varphi_t\in\mathcal{K}(\epsilon_0)$ for
any $t<T+\delta(T)$. We shall show that $\delta(T)$ must be the
infinity whenever $T$ and $N$ are large enough. First we prove

\begin{lem} Let $H(t)=\frac{1}{V}\int_M|\dot\varphi_t-c(t)|^2\omega_{\varphi_t}^n$.
 Then for any $t\in[0,T+\delta(T))$, there
exists a $\theta>0$ such that
 \begin{align}
 H(t)\le H(0)e^{-\theta t}.
\end{align}
\end{lem}

\begin{proof} For simplicity,  we let $\varphi=\varphi_t$. By (1.1),  one
sees
$$|\dot\varphi|\le 3\epsilon_0,~~\forall t\in[0,T+\delta(T)).$$
Since $\varphi$ satisfies,
\beqn\ddot\varphi=\triangle\dot\varphi+\dot\varphi,\eeqn
  then by a direct computation, we have
 \begin{align} &\frac{d}{dt}H_0(t)\notag\\
&=2\frac{1}{V}\int_M(\dot\varphi-c(t))(\ddot\varphi-\dot
c(t))\omega^n_\varphi
+\frac{1}{V}\int_M(\dot\varphi-c(t))^2\triangle_\varphi\dot\varphi\omega_\varphi^n\notag\\
&=2\frac{1}{V}\int_M(\dot\varphi_t-c(t))(\triangle_\varphi\dot\varphi+\dot\varphi)\omega^n_\varphi
+\frac{1}{V}\int_M(\dot\varphi-c(t))^2\triangle_\varphi\dot\varphi\omega_\varphi^n\notag\\
&=-2\frac{1}{V}\int_M|\nabla(\dot\varphi-c(t))|^2\omega^n_\varphi
+2\frac{1}{V}\int_M(\dot\varphi-c(t))^2\omega^n_\varphi\notag\\
&-2\frac{1}{V}\int_M(\dot\varphi-c(t))\|\nabla(\dot\varphi-c(t))\|^2\omega_\varphi^n\notag\\
&=2H_0(t)
-2\frac{1}{V}\int_M(1+\dot\varphi-c(t))\|\nabla(\dot\varphi-c(t))\|^2\omega_\varphi^n\notag\\
&\leq 2H_0(t)
-2(1-6\epsilon_0)\frac{1}{V}\int_M\|\nabla(\dot\varphi-c(t))\|^2\omega_\varphi^n.
\end{align}

Case 1, $\eta(M)=0$.  Then by (1.7) and  Lemma 1.1, we have
$$
\frac{d}{dt}H_0(t)\leq-[-2+2(1-6\epsilon_0)(1+\delta_0)]H_0(t).
$$
By choosing $\theta=-2+2(1-4\epsilon_0)(1+\delta_0)\ge \delta_0$,
we will get
\begin{align}
H_0(t)\leq H_0(0)e^{-\theta t}.
\end{align}

 Case 2,  $\eta(M)\neq0$. Since the Futaki-invariant vanishes,
for any $X\in \eta(M)$, we have
$$
\int_M\triangle(\theta_X+X(\varphi))(\dot\varphi-c(t))\omega^n_\varphi=0,
$$
 where $\theta_X$ is the potential of $X$ defined by (1.3) and $X(\phi)$ is the derivative of $\phi$ along $X$.
It follows
 $$|\int_M \theta_X(\dot\varphi-c(t))\omega^n_\varphi|\le
C\epsilon_0\int_M|\dot\varphi-c(t)|\omega^n_\varphi,$$
 for any $X\in\eta(M)$ with satisfying
 $\int_M\|X\|^2_{\omega_{KE}}\omega_{KE}^n=1.$
Here we used an estimate
$$ \|\varphi-\underline{\varphi}\|_{C^3}=O(\epsilon_0)$$
 by the regularity of  K\"ahler potentials
 $\varphi\in\mathcal{K}(\epsilon_0)$, which can be obtained by the
 Implicity Functional Theorem for equation (1.22) (cf. an argument at the last paragraph of this section) with the help of
 $W^{k,2}$-estimate
 $$\|\dot\varphi-c(t)\|_{W^{k,2}}= O(\epsilon_0)$$
  for $\dot\varphi$ (cf. the argument in Appendix 4.1). Thus by
using the continuity of the eigenvalue functions, one sees
  \begin{align}|\int_M
\psi^i(\dot\varphi-c(t))\omega^n_\varphi|\le
C'\epsilon_0\int_M|\dot\varphi-c(t)|\omega^n_\varphi,\end{align}
 where $\psi^i$ are the first eigenvalue functions of the Lapalacian
 operator associated to the metric $\omega_\varphi$,  which satisfy  $\int_M |\psi^i|^2\omega_{\varphi}^n=1.$

Let  $\Lambda_1(M,\omega_{\varphi})$ be a linear space  spanned by
a basis $\{\psi^i\}$ and $\Lambda_1^\bot(M,\omega_{\varphi})$ be a
subspace of $L^2$-integral functions   which are orthogonal to
$\Lambda_1(M,\omega_{\varphi})\cup\Bbb R$. Then we can decompose
$(\dot\varphi_t-c(t))$ as
$$
\dot\varphi_t-c(t)=\phi+\phi^\bot,$$
 with  $\phi\in \Lambda_1(M,\omega_{\varphi})$  and
$\phi^\bot\in\Lambda_1^\bot(M,\omega_{\varphi})$. Thus by (1.9),
we get
$$
\int_M|\phi|^2\omega^n_\varphi\leq A\epsilon_0^2
\int_M(\dot\varphi_t-c(t))^2\omega^n_\varphi,
$$
for some uniform constant $A$. It follows
\begin{align}
\int_M|\phi^\bot|^2\omega^2_\varphi\geq(1-A\epsilon_0^2)\int_M(\dot\varphi_t-c(t))^2\omega^2_\varphi.
\end{align}
Hence  by Lemma 1.1, we get
\begin{align}
\int_M \|\nabla(\dot\varphi_t-c(t)\|^2\omega^n_\varphi &\geq
\int_M\|\nabla\psi^\bot\|^2\omega^n_\varphi\notag\\
&\geq (1+\sigma_0)(1-A\epsilon_0^2)
\int_M(\dot\varphi_t-c(t))^2\omega^n_\varphi.
\end{align}
By choosing $\theta=
  2(1-6\epsilon_0)(1+\sigma_0)(1-C\epsilon_0^2)-2\ge
\sigma_0$, we obtain from (1.7),
$$
\frac{ d H_0(t)}{dt}\leq -\theta H_0(t).
$$
As a consequence, we have
$$
H_0(t)\leq H_0(0)e^{-\theta t}.
$$
\end{proof}

Next we want to use  Perelman's deep estimates for the gradient of
$\dot\varphi_t$ and the non-collapsing result for metric
$\omega_{\varphi_t}$ to get a $C^0$-estimate of $\dot\varphi_t$
with help of Lemma 1.3. Let's state the Perelman's result (a
detailed proof can be found in [ST]).

\begin{lem}(Perelman)
Let $g_t$ be the evolved  K\"ahler metrics  of  (0.1) and
$\varphi=\varphi_t$ be  K\"ahler potentials of $g_t$. Then there
exists a uniform constant $C$ independent of $t$ (just depending
on the initial metric $g$) such that  the following two facts
hold,

i) $\|\nabla \dot\varphi\|_{\omega_{\varphi}}\leq C$;

 ii) for $x\in M$
and $0<r\leq \text{diam}(M,g(t))$,
$\int_{B_r(x)}\omega^n>C^{-1}r^{2n}$,
 where  $\text{diam}(M,g(t))$ denote the diameters of $(M,g(t))$
which are  uniformly bounded.
\end{lem}

\begin{lem}\label{C0 Bound}
For any $t\in[0,T+\delta(T))$, we have
\begin{align}
|\dot\varphi_t-c(t)|\leq\min\{(\frac{\epsilon_0}{2N})^2,Ce^{-\frac{\theta}{2(n+1)}t}\}
\end{align} and
\begin{align}
||\dot\varphi_t-c(t)||_{C^\alpha}\leq
C\{\min\{(\frac{\epsilon_0}{2N})^2,
Ce^{-\frac{\theta}{2(n+1)}t}\}\}^{\frac{1}{2}}.
\end{align}
Here  $\alpha\le\frac{1}{4}$ and  $C$  depends only on the
constant in Lemma 1.4.
\end{lem}

\begin{proof} We suffice to consider the case of $\alpha=\frac{1}{4}$.
Let $x_0$ be the point where $|\tilde h|=|\tilde
h_t|=|\dot\varphi_t-c(t)|$ achieves its maximum. Choose a small
ball $B_r(x_0)$ for
$r=\min\{\text{diam}(M,g(t)),e^{-\frac{\theta}{2(n+1)}t}\}$. So we
have for any $x\in B_r(x_0)$,
\begin{align}
0\leq|\tilde h(x_0)|\leq |\tilde h(x)|+\|\nabla\tilde h\|r=|\tilde
h(x)|+\|\nabla h\|r.
\end{align}

Case 1), $e^{-\frac{\theta}{2(n+1)}t}\ge\text{diam}(M,g(t))$. Then
by Lemma 1.4, one sees
\begin{align}
\int_M |\tilde h(x_0)|^2\omega^n_\varphi & \leq 2\int_M|\tilde
h(x)|^2\omega^n_\varphi+ 2V\text{diam}(M,g(t))^2
\|\nabla \tilde h\|^2\notag\\
&\le Ce^{-\frac{\theta}{n+1}t}.\notag
 \end{align}
   Thus
 \beqn |\tilde h(x_0)|\le Ce^{-\frac{\theta}{2(n+1)}t}.\eeqn

Case 2), $e^{-\frac{\theta}{2(n+1)}t}\le\text{diam}(M,g(t))$. Then
\begin{align}
\frac{1}{V(B_r(x_0))}\int_{B_r(x_0)}|\tilde
h(x_0)|^2\omega^n_\varphi
&\leq\frac{2}{V(B_r(x_0))}\int_{B_r(x_0)}|\tilde h(x)|^2\omega^n_\varphi\notag\\
&+\frac{2}{V(B_r(x_0))}\int_{B_r(x_0)}\|\nabla \tilde h\|^2
r^2\omega^n_\varphi.\notag
\end{align}
Thus by Lemma 1.4, we get
\begin{align}
|\tilde h(x_0)|^2&\leq C e^{\frac{2n\theta}{2(n+1)}t}\int_M |\tilde h(x)|^2\omega^n_\varphi +Cr^2\notag\\
&\leq Ce^{-\frac{\theta}{n+1}t}.\notag
\end{align}
It follows
 \beqn |\tilde h(x_0)|\le C'e^{-\frac{\theta}{2(n+1)}t}.\eeqn
Therefore, both (1.15) and (1.16) give the estimate (1.12).

For any $x,y\in M$, by (1.12),  we have:  if
$\text{dist}(x,y)=\|x-y\|_{\omega_\varphi}\ge
e^{-\frac{\theta}{2(n+1)}t}$,
\begin{align}
&\frac{|\tilde h(x)-\tilde
h(y)|}{\|x-y\|^{\frac{1}{4}}_{\omega_{KE}}}\le 2 \frac{|\tilde
h(x)-\tilde
h(y)|}{\|x-y\|^{\frac{1}{4}}_{\omega_\varphi}}\notag\\
&\le  2|\tilde h(x)-\tilde h(y)|^{\frac{1}{2}}
\frac{ |\tilde h(x)-\tilde h(y)|^{\frac{1}{4}}_{\omega_\varphi}}{ \|x-y\|^{\frac{1}{4}}_{\omega_\varphi}}\notag\\
  &\le C\{\min\{(\frac{\epsilon_0}{2N})^2,Ce^{-\frac{\theta}{2(n+1)}t}\}\}^{\frac{1}{2}};
\end{align}
   if
$\text{dist}(x,y)\le e^{-\frac{\theta}{2(n+1)}t}$,
\begin{align}
&\frac{|\tilde h(x)-\tilde h
(y)|}{\|x-y\|^{\frac{1}{4}}_{\omega_{KE}}}\le 2 \frac{|\tilde
h(x)-\tilde h(y)|}{\|x-y\|^{\frac{1}{4}}_{\omega_\varphi}}\notag\\
&= 2\frac{|\tilde h(x)-\tilde h(y)|^{\frac{1}{2}}|\tilde
h(x)-\tilde h(y)|^{\frac{1}{2}}}
{\|x-y\|^{\frac{1}{2}}_{\omega_\varphi}}\|x-y\|^{\frac{1}{4}}_{\omega_\varphi}\notag\\
&\leq C|\tilde h|^{\frac{1}{2}}_{C^0}(\text{diam}(M,g(t)))^{\frac{1}{4}}\notag\\
&\leq
C'\{\min\{(\frac{\epsilon_0}{2N})^2,Ce^{-\frac{\theta}{2(n+1)}t}\}\}^{\frac{1}{2}}.
\end{align}
 Here we used Perelman's estimates again. (1.17) and (1.18) give the estimate (1.13).
\end{proof}

\begin{rem} We can  avoid to use Perelman's estimates to prove Lemma 1.5 by replacing to estimate the
$W^{k,2}$-norm of $\dot\varphi_t$. See   Appendix 1 in this paper.
\end{rem}

\begin{prop}    Choose some large $T$ such that
$$
C\frac{4(n+1)}{\theta}e^{-\frac{\theta}{2(n+1)}T}\leq\frac{\epsilon_0}{4N},
$$
where $C$ is  the constant chosen in Lemma 1.5. Then
\begin{align}
|\tilde \varphi|\leq\frac{3\epsilon_0}{4N}, ~~\forall
~t\in[0,T+\delta(T)),
\end{align}
 where
 $\tilde\varphi=\tilde\varphi_t=\varphi(t)-\frac{1}{V}\int_M\varphi\omega_\varphi^n.$
\end{prop}

\begin{proof}
Notice that
$$
\frac{d}{dt}\tilde\varphi=\tilde h-\frac{1}{V}\int_M\tilde
h\triangle_\varphi\varphi\omega_\varphi^n.
$$
Then by Lemma 1.5, we have
\begin{align}
\tilde\varphi&=\tilde\varphi_T+\int_T^{T+\delta(T)}\tilde h dt
-\int_T^{T+\delta(T)}\frac{1}{V}\int_M\tilde h\triangle_\varphi\varphi\omega_\varphi^ndt\notag\\
&\leq\frac{\epsilon_0}{2N}+C\int_T^{T+\delta(T)}e^{-\frac{\theta}{2(n+1)}t}dt+
2C\epsilon_0\int_T^{T+\delta(T)}e^{-\frac{\theta}{2(n+1)}t}dt\notag\\
&\leq\frac{\epsilon_0}{2N}+2C\frac{2(n+1)}{\theta}e^{-\frac{\theta}{2(n+1)}T}\notag.
\end{align}
\end{proof}

\begin{proof} [Proof of Theorem 1.1]
 First we want to show that $\varphi_t\in\mathcal{K}(\epsilon_0)$
for any $t>0$.  By the contradiction, we may assume that there
exists a number $\delta(T)<\infty$ such that
$\varphi_t\in\mathcal{K}(\epsilon_0)$ for any $t<T+\delta(T)$ and
there exists a sequence of $t_i\to T+\delta(T)$ such that
\begin{align}\label{assume}
\|\overline{\varphi_{t_i}}\|_{C^{2,\alpha}}=\|\varphi_{t_i}-\underline{\varphi_{t_i}}\|_{C^{2,\alpha}}\rightarrow\epsilon_0.
\end{align}
 Let $b_t$ be a constant  so that
$\overline\varphi=\tilde\varphi+b_t$. Then by  Proposition 1.7, it
is easy to see  $b_t\le\frac{2\epsilon_0}{N}$. Decompose
$\overline\varphi$ by $\overline\varphi=\phi+\phi^\bot$, where
$\phi\in\Lambda_1(M,\omega_{KE})$ and
$\phi^\bot\in\Lambda_1^\bot(M,\omega_{KE})$, where
$\Lambda_1^\bot(M,\omega_{KE})$ is a subspace of $L^2$-integral
functions   which are orthogonal to $\Lambda_1(M,\omega_{KE})\cup
\Bbb R$. Thus $\phi=\sum_i a_i\theta_i$ for some constants $a_i$,
where $\theta_i$ is a basis of the space
$\Lambda_1(M,\omega_{KE})$.  As a consequence, by Proposition 1.7,
we have $|a_i|\le\frac{2\epsilon_0}{N}$, so
 \beqn \|\phi\|_{C^{2,\alpha}}\le \frac{A_0\epsilon_0}{N},\eeqn
for some uniform constant $A_0$.

By equation (1.1), we have
\begin{align}
\omega_\varphi^n=\omega_{KE}^ne^{\tilde{h}+\overline\varphi+a},
\end{align}
 where $\tilde h=\dot\phi_t-c_t$  and $a=a_t$ are constants. By Lemma 1.5 and Proposition 1.7, it is easy to see
 that $|a|\le \frac{4A_0\epsilon_0}{N}$. Let $P$ be a projection from Banach space
 $H^{2,\alpha}(M)$ to Banach space
  $H^\alpha(M)\cap \Lambda_1^\bot(M,\omega_{KE})$.
  Then    $\phi^\bot$ is a solution of equation
 $$P[\log(\frac{[\omega_{\phi+\phi^\bot}]^n}{\omega_{KE}^n})]-\phi^\bot=P(\tilde
 h+a),$$
  where $\phi$ and $\tilde h+a$ are regarded as two perturbation functions.
  On the other hand,
  by Lemma 1.5, we have
  $$\|P(\tilde
 h+a)\|_{C^\alpha}=\|P(\tilde
 h)\|_{C^\alpha}\le
 C
 \min\{(\frac{2\epsilon_0}{N})^2,Ce^{-\frac{\theta}{2(n+1)}t}\}^{\frac{1}{2}}.$$
  Thus by using  the Implicity Functional Theorem, we get
\begin{align}
\|\phi^\bot\|_{C^{2,\alpha}}\leq c=O(\frac{\epsilon_0}{N}),
\end{align}
where constant $c$ is independent  of $t$ and $\epsilon_0$ and
goes to zero as $N\to \infty$. Consequently,  $c\le
\frac{\epsilon_0}{4}$ by choosing a large number $N$. Hence by
combining (1.21) and (1.23), we obtain
\begin{align}
\|\overline\varphi\|_{C^{2,\alpha}}\leq\frac{\epsilon_0}{2}.\forall
~~t\in[T,T+\delta(T)).
\end{align}
But this is impossible according to (1.20). Therefore we prove
that $\varphi_t\in\mathcal{K}(\epsilon_0)$ for any $t>0$.

By the above argument and lemma 1.5 and  Proposition 1.7, we
conclude that there exists an $\epsilon$ such that  if
$\|\psi-\underline{\psi}\|_{C^{2,\alpha}}\le \epsilon$, then for
any $t>0$, we have
 \begin{align} &\text{a)}~~\varphi_t\in\mathcal{K}(\epsilon_0),\\
&\text{b)}~~|\tilde\varphi|\le\frac{3\epsilon_0}{4N},\\
&\text{c)}~~\|\tilde h\|_{C^\alpha}\leq
C\{\min\{(\frac{\epsilon_0}{2N})^2,Ce^{-\frac{\theta}{2(n+1)}t}\}\}^{\frac{1}{2}}.
\end{align}
  On the other hand,  according to
 [BM], one can choose an element $\sigma_t\in\text{Aut}_0(M)$ for each $\varphi$ such
 that  potential $(\varphi_{\sigma}-\underline{\varphi_{\sigma}})$
lies in $\Lambda_1^\bot(M,\omega_{KE})$, where
$\varphi_{\sigma}=\varphi_{\sigma_t}=\varphi_t(\sigma_t(.))+\rho_t(.)$
and $\rho_t$ is K\"ahler potential defined by (1.2). Furthermore,
by the fact $\varphi\in \mathcal{K}(\epsilon_0)$, one can prove
easily
$$\text{dist}(\sigma,Id)\le 1.$$
Consequently, by (1.27), we have
$$\|\tilde h(\sigma_t(.))\|_{C^{\alpha}}\le C e^{-\frac{\theta}{(n+1)}t}.$$
Thus by applying  the Implicity Functional Theorem to the modified
equation of (1.22),
$$\omega_{\varphi_\sigma}^n=\omega_{KE}^ne^{\tilde
h(\sigma_t(.))-\varphi_\sigma+a},$$
 we have
$$\|\varphi_{\sigma}-\underline{\varphi_{\sigma}}\|_{C^{2,\alpha}}\leq C(\|\tilde
h(\sigma_t(.))\|_{C^{\alpha}}).$$
 Furthermore, one can get an explicit estimate
$$\|\varphi_{\sigma}-\underline{\varphi_{\sigma}}\|_{C^{2,\alpha}}\leq 2\|\tilde
h(\sigma_t(.))\|_{C^{\alpha}}\le C' e^{-\frac{\theta}{(n+1)}t}.$$

To get higher-order estimates for the modified K\"ahler potentials
$\varphi_\rho$, one can use Lemma 4.1 in Appendix 1 and the
embedding theory of Sobolev spaces to obtain
 $$\|\tilde{\dot\varphi}\|_{C^{k,\alpha}}\le C_k
 e^{-\frac{\theta}{n+1}t},~~\forall ~t>0,$$
 where constants $C_k$ depends only on $k, \epsilon_0$ and higher-order derivatives of the initial
 K\"ahler potential $\psi$ (we may assume that $\psi$ is smooth since we can replace it by
 an evolved K\"ahler metric $\varphi_{t=1}$). Then by the Implicity Functional Theorem as the above,
we derive
$$\|\varphi_{\sigma}-\underline{\varphi_{\sigma}}\|_{C^{k+2,\alpha}}\leq 2\|\tilde
{\dot\varphi}\|_{C^{k,\alpha}}\le C_k' e^{-\frac{\theta}{n+1}t}.$$
Therefore we prove that K\"ahler metrics
$\sigma^*(\omega_{\varphi})$ converge exponentially to
$\omega_{KE}$.

\end{proof}

\section{In case of K\"ahler-Ricci soliton}

In this section, we assume that $M$ admits a K\"ahler Ricci
soliton $(\omega_{KS}, X_0)$  with some holomorphic vector field
$X_0$ on $M$, i.e., $(\omega_{KS},X_0)$ satisfies equation,
$$
\text{Ric}(\omega_{KS})-\omega_{KS}=L_{X_0}\omega_{KS},$$
 where $L_{X_0}$ is a Lie derivative along the vector field $X_0$.
By the Hodge theorem, one can define  a real-valued potential
$\theta_X$ of $X_0$ by
\begin{align} &L_{X_0}\omega_{KS}=\sqrt{-1}\partial\overline\partial\theta_X,\notag\\
&\int_M e^{\theta_X}\omega_{KS}^n=\int_M\omega_{KS}^n.\notag
\end{align}
So if we let $X$ and $X'$  be a real part and imaginary part of
$X_0$, respectively, then for any $\phi\in
\mathcal{M}(\omega_{KS})$, we  have
\begin{align}
L_{X_0}\omega_{\phi}=\sqrt{-1}\partial\overline\partial(\theta_X+X_0(\phi)),
\end{align}
 and so
$$L_{X}\omega_{\phi}=\sqrt{-1}\partial\overline\partial(\theta_X+X(\phi))$$
 and
$$L_{X'}\omega_{\phi}=\sqrt{-1}\partial\overline\partial(X'(\phi)).$$
  (2.1) also implies that for any $\psi\in C^\infty(M)$ it holds
$$<\overline\partial( \theta_X+X_0(\phi)),
\overline\partial\psi>_{\omega_\phi}=X_0(\psi)=X(\psi)+\sqrt{-1}X'(\psi).$$
 Thus
 \begin{align}|<\nabla(
 \theta_X+X(\phi)),\nabla\psi>_{\omega_\phi}-X(\psi)|\le
 |X'(\phi)|\|\nabla\psi\|_{\omega_\phi}.\end{align}

 We  now consider a modified equation of (0.1),
\begin{align}
&\frac{\partial g(t,\cdot)}{\partial t}=-\text{Ric}(g)+g+L_X g,\notag\\
&g(0)=g.
\end{align}
Then  (2.3) is equivalent to  a parabolic equation of complex
Monge-Amp\`ere type,
\begin{align}
&\frac{\p\varphi}{\p
t}=\log\frac{\omega^n_\varphi}{\omega_{KS}^n}+\varphi+X(\varphi), \notag\\
&\varphi(0)=\psi-\underline{\psi},
\end{align}
where $\underline\psi=\frac{1}{V}\int_M\psi\omega_{KS}^n$,
$V=\int_M\omega_{KE}^n$, and $\varphi=\varphi_t$ are potentials of
evolved K\"ahler metrics $g_t$ of (2.3).  Let $K_{X_0}$ be an
one-parameter compact subgroup of $\text{Aut}_0(M)$ generated by
the imaginary part $X'$ of $X_0$. By choosing a reductive subgroup
$\text{Aut}_r(M)$ of $\text{Aut}_0(M)$ such that $\text{Aut}_r(M)$
contains $K_{X_0}$, we can prove

\begin{theo} Let $M$ be  a compact K\"ahler manifold $M$ with
$c_1(M)>0$ which admits a K\"ahler-Ricci soliton $\omega_{KS}$.
Then there exists a small $\epsilon$ such that for  any initial
data, potential $\psi\in\mathcal{M}(\omega_{KS})$ in equation
(2.4) with $\psi\in \mathcal{K}(\epsilon)$, there exist a family
of $\sigma=\sigma_t\in\text{Aut}_r(M)$ for evolved K\"ahler
potentials $\varphi=\varphi_t$ of (2.4) at $t$  such that K\"ahler
potentials $(\varphi_\sigma-\underline{\varphi_\sigma})$ are
$C^k$-norm uniformly bounded, where
$\varphi_\sigma=\sigma^*\varphi+\rho$ and $\rho=\rho_t$ are
K\"ahler potentials defined by
 \begin{align}&\rho^*(\omega_{KS})=\omega_{KS}+\sqrt{-1}\partial\overline\partial\rho,\notag\\
 &\int_M e^{-\rho-X_0(\rho)}\omega_{KS}^n=\int_M\omega_{KS}^n.
 \end{align}
As a consequence, evolved K\"ahler metrics $g_t$ of (2.3) converge
to $g_{KS}$ smoothly in the sense of Cheeger-Gromov. Furthermore,
if in addition that $\psi$ is $K_{X_0}$-invariant, then  there
exist a family of $\sigma=\sigma_t\in\text{Aut}_r(M)$ such that
$(\varphi_\sigma-\underline{\varphi_\sigma})$ converge
exponentially to $0$ as $t\to \infty$, and consequently K\"ahler
metrics $\sigma^*(\omega_{\varphi})$ converge exponentially to
$\omega_{KS}$.
\end{theo}

As in Section 1, to prove Theorem 2.1,  we need to estimate
$\dot\varphi$ of K\"ahler potentials $\varphi=\varphi_t$ of (2.4).
We introduce a modified functional of $H_0(t)$   by
$$\tilde H_0(t)=\frac{1}{V}\int_M(\dot\varphi-c(t))^2
e^{\tilde h}\omega^n_\varphi,$$
 where  $c(t)=\int_M\dot\varphi e^{\tilde h}\omega_\varphi^n$ is a constant, $\tilde h=\tilde h_t=\theta_X+X(\varphi)-\dot\varphi$
  and  $V=\int_M\omega_{KS}^n$.
 By  a direct computation, one shows
\begin{align}
 &\frac{d}{dt}\tilde H_0(t)
=2\frac{1}{V}\int_M(\dot\varphi-c(t))(\ddot\varphi
-\dot c(t))e^{\tilde h}\omega^n_\varphi\notag\\
&+\frac{1}{V}\int_M(\dot\varphi-c(t))^2(\triangle_\varphi\dot\varphi+X(\dot\varphi)-\ddot\varphi)e^{\tilde h}\omega_\varphi^n\notag\\
&=2\frac{1}{V}\int_M(\dot\varphi-c(t))(\triangle_\varphi\dot\varphi+\dot\varphi+X(\dot\varphi)-c(t))
e^{\tilde h}\omega^n_\varphi\notag\\
&+\frac{1}{V}\int_M(\dot\varphi-c(t))^2(\triangle_\varphi\dot\varphi+X(\dot\varphi)-\ddot\varphi)e^{\tilde h}\omega_\varphi^n\notag\\
&=2\frac{1}{V}\int_M
[\dot\varphi-c(t)](\triangle_\varphi\dot\varphi+X(\dot\varphi))e^{\tilde
h} \omega^n_\varphi\notag\\
&+2\frac{1}{V}\int_M(\dot\varphi-c(t))^2(1+\dot\varphi)e^{\tilde
h}\omega^n_\varphi.
\end{align}
On the other hand, by (2.2), we see
 \beqs &&\int_M
[\dot\varphi-c(t)](\triangle_\varphi\dot\varphi+X(\dot\varphi))e^{\tilde
h} \omega^n_\varphi\\
&&=-\int_M \|\nabla(\dot\varphi-c(t))\|^2 e^{\tilde h}
\omega^n_\varphi\notag\\
&&+ \int_M
[\dot\varphi-c(t)][X(\dot\varphi)-<\nabla(\theta_X+X(\phi)-\dot\varphi),\nabla\dot\varphi>]e^{\tilde
h} \omega^n_\varphi\\
&&\le -\int_M \|\nabla(\dot\varphi-c(t))\|^2 e^{\tilde h}
\omega^n_\varphi\notag+ V|\dot\varphi-c(t)|
|X'(\dot\varphi)|\|\nabla{\dot\varphi}\|. \eeqs
 Thus inserting the above inequality into (2.6), we get
\begin{align}
& \frac{d}{dt}\tilde H_0(t) \le
2\frac{1}{V}\int_M(\dot\varphi-c(t))^2(1+\dot\varphi)e^{\tilde
h}\omega^n_\varphi\notag \\
& -2\frac{1}{V}\int_M \|\nabla(\dot\varphi-c(t))\|^2 e^{\tilde h}
\omega_\varphi^n\notag\\
&+ 2|\dot\varphi-c(t)| |X'(\dot\varphi)|\|\nabla{\dot\varphi}\|.
 \end{align}

We  shall estimate the $L^2$-integral of $\nabla\dot\varphi$ and
need the following lemma,

\begin{lem} Let $P=(P,\omega_\phi)$ be an elliptic operator on  $C^{k,\alpha}(M)$ defined by
$$P\psi=\triangle \psi+\psi+\text{Re}<\overline\partial{ h}, \overline\partial\psi>_{\omega_\phi},$$
where $\triangle$ is the Lapalace operator with respect to a
K\"ahler metric $\omega_{\phi}$ and $h$ is a Ricci potential of
$\omega_{\phi}$. Then $\text{ker}(P,\omega_\phi)
\subset\eta_r(M),$
 where $\eta_r(M)$ is a reductive part of Lie algebraic $\eta(M)$
consisting of all holomorphic vector fields on $M$. Moreover, if
$\omega_{\phi}=\omega_{KS}$, then $\text{ker}(P,\omega_{KS})
\cong\eta_r(M)$.
\end{lem}

\begin{proof} Let $L\psi=\triangle \psi+\psi+<\overline\partial{ h}, \overline\partial\psi>_{\omega_\phi}$
 and $\overline L\psi=\triangle
\psi+\psi+\overline{<\overline\partial{ h},
\overline\partial\psi>_{\omega_\phi}}$, where $h$ is a Ricci
potential of the metric $\omega_\phi$. Then by the Bochner
formula, one can show (cf. Lemma 3.1 in [TZ3]),
 \begin{align} \int_M -(L\psi)\psi e^{
h}\omega_{\phi}^n=\int_M (\|\nabla\psi\|^2-\psi^2)e^{
h}\omega_{\phi}^n\ge 0,
\end{align}
 and
\begin{align}
 \int_M -(\overline L\psi)\psi e^{
h}\omega_{\phi}^n=\int_M (\|\nabla\psi\|^2-\psi^2)e^{
h}\omega_{\phi}^n\ge 0.\end{align}
 Moreover, the equality (2.8) or (2.9) holds if and only if
the corresponding vector field of  $(0,1)$-form
$\overline\partial\phi$ is holomorphic.
 Thus
 \begin{align} &-2\int_M (P\psi)\psi e^{ h}\omega_{\phi}^n
 =-\int_M (L\psi+\overline L\psi)\psi e^{ h}\omega_{\phi}^n\notag\\
 &=2\int_M (\|\nabla\psi\|^2-\psi^2)e^{ h}\omega_{\phi}^n\ge 0,
 \end{align}
and the equality holds if and only if the corresponding vector
field of  $(0,1)$-form $\overline\partial\psi$ is holomorphic.
Since $\psi$ is a real-valued function the corresponding vector
field must lie in $\eta_r(M)$. Furthermore, if one defines a
potential $\theta_Y'$ by
 $$L_{Y}\omega_{KS}=\sqrt{-1}\partial\overline\partial\theta_Y'~~\text{and}~~
\int_M \theta_Y'e^{\theta_X}\omega_{KS}^n=0,
$$
for  an element $Y$ in $\eta_r(M)$,   then in case of
$\omega_{\phi}=\omega_{KS}$, by using the fact that $X_0$ is an
element of center of $\eta_r(M)$ [TZ1] and $h=\theta_{X_0}$, one
can show $\theta_Y'$ must be in $\text{ker}(P,\omega_{KS})$.
\end{proof}

Set a Banach space by
$$\overline{\mathcal{K}}(\epsilon_0)=\{\phi\in\mathcal{M}(\omega_{KS})|~~\|\phi-\underline\phi\|_{C^{2,\alpha}}\leq\epsilon_0\}.
$$
Then we have

\begin{lem} Let $\varphi=\varphi_t$ be an evolved K\"ahler potential of (2.4) at $t$
and $\theta_Y'\in \text{ker}(P,\omega_\varphi)$ be a potential of
$Y\in\eta_r(M)$ with $\int_M \|Y\|^2\omega_{KS}^n=1$.  If
$\varphi\in\overline{\mathcal{K}}(\epsilon_0)$,  then there exist
two uniform constants $C_1$ and $C_2$ such that
\begin{align}&|\int_M \theta_Y'
(\dot\varphi_t-c(t))e^{\theta_X+X(\varphi)}\omega_\varphi^n|\notag\\
& \le
C_1\epsilon_0\int_M|\dot\varphi_t-c(t)|e^{\theta_X+X(\varphi)}\omega_\varphi^n+C_2\epsilon_0^2.
\end{align}

\end{lem}

\begin{proof}
Recall a generalized Futaki-invariant defined in [TZ2] by
$$F_{X_0}(Y)=\int_M
Y[h_{\omega_\phi}-(\theta_X+X_0(\phi))]e^{\theta_X+X_0(\phi))}\omega_\phi^n,~~\forall~
Y\in\eta(M).$$
 It was proved that the invariant is independent of the choice of
 K\"ahler metric $\omega_\phi$ on $M$ and the invariant vanishes if $M$ admits the K\"ahler-Ricci soliton $(\omega_{KS},
 X_0)$. So we have
 $$ F_{X_0}(Y)\equiv 0,~~\forall~ Y\in\eta(M).
  $$
By applying the metrics $\omega_\varphi$ to the above identity,
one sees
$$\int_M Y[\dot\varphi-c(t)-\sqrt{-1}X'(\varphi)]e^{\theta_X+X_0(\varphi))}\omega_\varphi^n=0,~~\forall~Y\in\eta_r(M).$$
 It follows
\begin{align} &|\text{Re}(\int_M
Y[\dot\varphi-c(t)-\sqrt{-1}X'(\varphi)]e^{\theta_X+X(\varphi)+
\cos(X'(\varphi))}\omega_\varphi^n)|\notag\\
&\le A_0\|\varphi\|_{C^2}|X'(\varphi)|.\end{align}
 On the other hand, by using  the Stoke's formula, we have
\begin{align} &\int_M
Y[\dot\varphi-c(t)-\sqrt{-1}X'(\varphi)]e^{\theta_X+X(\varphi)+
\ln\cos(X'(\varphi))}\omega_\varphi^n\notag\\
&=-\int_M(\dot\varphi-c(t))-\sqrt{-1}X'(\varphi))\notag\\
&\times
[\triangle(\theta_Y'+Y(\varphi))+<\overline\partial(\theta_Y'+Y(\varphi)),\overline\partial(\theta_X+X(\varphi)+
\ln \cos(X'(\varphi)))>]\notag\\
&\times e^{\theta_X+X(\varphi)+ \ln\cos(X'(\varphi))}\omega_\varphi^n\notag\\
&=\int_M(\dot\varphi-c(t)-\sqrt{-1}X'(\varphi))(\triangle\theta_Y'+<\overline\partial\theta_Y',\overline\partial\theta_X>)\notag\\
&\times e^{\theta_X+X(\varphi)+
\ln\cos(X'(\varphi))}\omega_\varphi^n+O(\epsilon_0^2)\notag\\
&=\int_M(\dot\varphi-c(t)-\sqrt{-1}X'(\varphi))(\triangle_{\omega_{KS}}\theta_Y'+<\overline\partial\theta_Y',
\overline\partial\theta_X>_{\omega_{KS}})\notag\\
&\times e^{\theta_X+X(\varphi)+ \ln\cos(X'(\varphi))
}\omega_\varphi^n+O(\epsilon_0^2).
\end{align}
Note that
$$<\overline\partial\theta_Y',\overline\partial\theta_X>_{\omega_{KS}}=Y(\theta_X)=X(\theta_Y')=<\overline\partial\theta_X,
\overline\partial\theta_Y'>_{\omega_{KS}}$$
 is a real-valued function ([TZ1]). Thus
$$\triangle_{\omega_{KS}}\theta_Y'+<\overline\partial\theta_Y',
\overline\partial\theta_X>_{\omega_{KS}}=\triangle_{\omega_{KS}}\theta_Y'+<\overline\partial\theta_X,
\overline\partial\theta_Y'>_{\omega_{KS}}=-\theta_Y'.$$
 Therefore, inserting  (2.13) into (2.12),  one will get (2.11).
\end{proof}

 By using Lemma 2.2 and Lemma 2.3,  we can complete  the   $L^2$-estimate of $\dot\varphi$.

\begin{lem} Let $\epsilon_0<<1$. Then
 \begin{align}
 \tilde H_0(t)\le \tilde H_0(0) e^{-\theta t}+ \frac{B_0}{\theta}\epsilon_0^3, ~~\forall ~~t\in [0,T),
 \end{align}
 if $\varphi_t$ lies in
 $\overline{\mathcal{K}}(\epsilon_0)$  and  $\tilde H_0(t)\ge\frac{B_0}{\theta}\epsilon_0^3$ for any $t$ in $[0,T)$,
where the constant $B_0=B_0(\|X'\|_{C^0})$ depends only on
$\|X'\|_{C^0}$ and the constant $\theta>0$  depends only on
  the gap of the first two eigenvalues of the operator $P$ associated to the metric $\omega_{KS}$ in
  Lemma 2.2.
\end{lem}

\begin{proof}
Let $\psi^i$ be the first eigenvalue functions of the
 operator $(P,\omega_{\varphi})$ with respect to the  metric $\omega_\varphi$
 with satisfying $\int_M |\psi^i|^2 e^{\tilde h} \omega_{\varphi}^n=1.$
Then by the continuity of eigenvalue functions and (2.11), one
sees that there exists two constants $C$ and $A_0$ such that
 \begin{align} &|\int_M \psi^i
(\dot\varphi_t-c(t))e^{\tilde h}\omega_\varphi^n\notag\\
 &\le
C\epsilon_0\int_M|\dot\varphi_t-c(t)|e^{\tilde
h}\omega_\varphi^n+A_0\epsilon_0^2.
\end{align}
Now  as same as in  the proof of Lemma 1.3,  we decompose
$\dot\varphi_t-c(t)$ as $\psi+\psi^\bot$ with  $\psi\in
\Lambda_1(M,\omega_{\varphi})$  and
$\psi^\bot\in\Lambda_1^\bot(M,\omega_{\varphi})$, where
$\Lambda_1(M,\omega_{\varphi})$ is a linear space spanned by a
basis  $\{\psi^i\}$ and $\Lambda_1^\bot(M,\omega_{\varphi})$ be a
subspace of $L^2$-weighted integral functions   which are
orthogonal to $\Lambda_1(M,\omega_{\varphi})\cap\Bbb R$ in the
sense of
 $$ \int_M \psi\psi'e^{\tilde h} \omega_{\varphi}^n=0,~~\forall~\psi\in\Lambda_1(M,\omega_{\varphi}),
  \psi'\in\Lambda_1^\bot(M,\omega_{\varphi}).$$
   Then  we get
$$
\int_M|\psi|^2e^{\tilde h}\omega^n_\varphi\leq C'\epsilon_0^2
\int_M(\dot\varphi-c(t))^2e^{\tilde h}\omega^n_\varphi +
nA_0^2\epsilon_0^4,
$$
 and so
\begin{align} &\int_M|\psi^\bot|^2e^{\tilde h}\omega^n_\varphi\notag\\
&\ge (1-C'\epsilon_0^2) \int_M(\dot\varphi-c(t))^2e^{\tilde
h}\omega^n_\varphi- nA_0^2\epsilon_0^4.
 \end{align}
 On the other hand, by using the continuity of eigenvalues and Lemma 2.2,
  there exists a number $\delta_0>0$ (compared to Lemma 1.2),
  which  depends only on
  the gap of the first two eigenvalues of the operator $(P,\omega_{KS})$ with respect to the metric $\omega_{KS}$ in
  Lemma 2.2, such that for any
  $\varphi\in\overline{\mathcal{K}}(\epsilon_0)$,  we have
$$\int_M\|\nabla\psi^\bot\|^2 e^{\tilde h}\omega^n_\varphi\notag\\
\geq (1+\delta_0) \int_M(\psi^\bot)^2 e^{\tilde
h}\omega^n_\varphi.$$
 Thus by (2.16), we get
\begin{align}
&\int_M \|\nabla(\dot\varphi_t-c(t))\|^2e^{\tilde h}
\omega^n_\varphi &\notag\\
&\geq (1+\delta_0)(1-C'\epsilon_0^2)
\int_M(\dot\varphi_t-c(t))^2e^{\tilde h}\omega^n_\varphi -
nA_0^2\epsilon_0^4.
\end{align}
By inserting (2.17) into (2.7), we obtain
\begin{align}
&\frac{ d \tilde H_0(t)}{dt}\notag\\
&\leq -
2[(1-2\epsilon_0)(1+\delta_0)(1-C'\epsilon_0^2)-(1+\epsilon_0)]
\tilde H_0(t)+ B_0\epsilon_0^3\notag\\
&\leq -\theta \tilde H_0(t)+ B_0\epsilon_0^3,
\end{align}
 where  the constant $B_0=B_0(\|X'\|_{C^0})$ depends only on
$\|X'\|_{C^0}$ and $\theta=
2[(1-2\epsilon_0)(1+\delta_0)(1-C'\epsilon_0^2)-(1+\epsilon_0)]\ge\delta_0$
as  $\epsilon_0$ is  small enough.

By (2.18), we have
$$
\frac{ d (\tilde H_0(t)-\frac{B_0\epsilon_0^3}{\theta})}{dt}\le
-\theta(\tilde H_0(t)-\frac{B_0\epsilon_0^3}{\theta}).$$
Since  $\tilde H_0(t)\ge\frac{B_0\epsilon_0^3}{\theta}$, we get
 \beqs
\tilde H_0(t)\le & e^{-\theta t}( \tilde
H_0(0)-\frac{B_0\epsilon_0^3}{\theta})+
\frac{B_0\epsilon_0^3}{\theta}\\
&\le e^{-\theta t} \tilde H_0(0)+ \frac{B_0\epsilon_0^3}{\theta}.
\eeqs

\end{proof}

\begin{rem} From (2.7) and (2.12),  we see that if in
addition that the initial K\"ahler potential  $\psi$ is
$K_{X_0}$-invariant, then (2.11) can be  improved  as
$$\tilde H_0(t)\le \tilde H_0(0) e^{-\theta t}, ~~\forall ~~t\in [0,T)$$
  whenever  $\varphi_t$ lies in
 $\overline{\mathcal{K}}(\epsilon_0)$.
\end{rem}

To get a  $C^0$-estimate and $C^\alpha$-estimate for
$\dot\varphi$, we use a method as in Appendix  to estimate
$W^{k,2}$-estimates ($k\ge 1$) for $\dot\varphi$. Let
 $$\tilde H_k(t)=\int_M\|\nabla^k\dot\varphi\|^2e^{\tilde h}\omega_{\varphi}^n.$$
Then  we have

\begin{prop} Let $\epsilon_0<<1$. Then under the same condition in Lemma 2.4,    there exist
 two uniform constants $\theta', B>0$ which depend only on the metric $\omega_{KS}$ and integer number $k$   such that
 \begin{align} \tilde H_k(t)\le
 e^{-\theta' t} (\tilde H_k(0)+B\tilde H_0(0))+\frac{B_0B}{\theta'}\epsilon_0^3, ~~\forall ~t\in [0,T),\end{align}
if $\tilde H_k(t)+B\tilde H_0(t))\ge
\frac{B_0B}{\theta'}\epsilon_0^3$ for any $t\le T$,
 where $B_0$ is the constant determined in Lemma 2.4.
 \end{prop}

\begin{proof} First note that similarly to (4.1) in Appendix 1, we
can obtain
\begin{align}
&\frac{d\|\nabla^k \dot\varphi\|^2}{dt}\notag\\
& \le
-2\|\nabla^{k+1}\dot\varphi\|^2+C_1\|\nabla^k\dot\varphi\|^2+C_2\|\dot\varphi-c(t)\|^2.\notag
\end{align}
It follows
\begin{align}
\frac{ d \tilde H_k(t)}{dt}&=\int_M
\frac{d\|\nabla^k\dot\varphi\|^2}{dt}e^{\tilde
h}\omega_{\phi}^n+\int_M
\|\nabla^k\dot\varphi\|^2(\triangle\dot\varphi+X(\dot\varphi)-\ddot\varphi) e^{\tilde h} \omega_{\varphi}^n\notag \\
&\le -2\tilde H_{k+1}(t)+C_1'\tilde H_k(t)\notag+C_2'\|\dot\varphi-c(t)\|^2\notag\\
&\le -\theta' \tilde H_{k}(t)+C_3\tilde H_0(t).\notag
\end{align}
On the other hand, by (2.18), we  have
$$
\frac{d \tilde H_0(t)}{dt}\le  -\theta \tilde H_0(t)+
B_0\epsilon_0^3, ~~\forall ~t\in [0,T),$$
 since we may also assume that $\tilde H_0(t)\ge\frac{B_0}{\theta}\epsilon_0^3$ for any $t$ in $[0,T)$,
Thus  combining the above  two inequalities, we get
 \begin{align} \frac{
d( \tilde H_k(t)+B\tilde H_0(t))}{dt}&\le -\theta'
[\tilde H_{k}(t)+\frac{(B\theta-C_3)}{\theta'}\tilde H_0(t)]+B_0B\epsilon_0^3\notag\\
&\le -\theta' (\tilde H_k(t)+B\tilde H_0(t))+B_0B\epsilon_0^3,
 \end{align}
  where $B$ is a sufficiently  large number independent of $\epsilon_0$.
From (2.20), one can easily get
 $$ (\tilde H_k(t) +B\tilde H_0(t)) \le
 e^{-\theta' t} (\tilde H_k(0)+B\tilde H_0(0))+\frac{B_0B}{\theta'}\epsilon_0^3, ~~\forall ~t\in [0,T),$$
  and so (2.19) is true.

\end{proof}

By the embedding theory of Sobolev spaces, we get

\begin{cor}  Let $\epsilon_0<<1$. Then under the same condition in Lemma 2.4, there exist
 two uniform constants $\theta_0, C_0>0$ which depend only on the metric
 $\omega_{KS}$ such that
\begin{align}\|\tilde{\dot{\varphi_t}}\|_{C^{\alpha}}\le  C_0[e^{-\theta_0
t}\|\psi-\underline\psi\|_{C^{2,\alpha}}+
\epsilon_0^{\frac{3}{2}}],
 ~~\forall ~t\in [0,T),\end{align}
if $\|\tilde{\dot{\varphi_t}}\|_{C^{\alpha}}\ge
C_0\epsilon_0^{\frac{3}{2}}$ for any $t\le T$.
\end{cor}

\begin{rem} By Remark 2.5, according to the proof of Proposition 2.6, we see that if in
addition that the initial K\"ahler potential  $\psi$ is
$K_{X_0}$-invariant, then (2.19) can be  improved  as
$$\tilde H_k(t)\le (\tilde H_0(0)+B\tilde H_k(0)) e^{-\theta' t}, ~~\forall ~~t\in [0,T)$$
  whenever  $\varphi_t$ lies in
 $\overline{\mathcal{K}}(\epsilon_0)$. Thus (2.21) can be improved  as
$$\|\tilde{\dot{\varphi_t}}\|_{C^{\alpha}}\le  C_0e^{-\theta_0
t}\|\psi-\underline{\psi}\|_{C^{2,\alpha}}.$$
\end{rem}

The following lemma can be easily proved by using apriori
estimates for solution $\varphi(t,\cdot)$ of (2.4) at finite time
(cf [TZ3]).

\begin{lem} Let $\psi\in\mathcal{K}(\frac{\epsilon_0}{N})$. Then there exists
$T=T_N$ such that evolved K\"ahler potentials  $\varphi_t$ of
(2.4) with $\psi$ as an initial potential  lies in
$\mathcal{K}(\epsilon_0)$ for any $t<T$.
\end{lem}

We are now going to do a key estimate for the  proof of Theorem
2.1.

\begin{prop} There exist a small $\epsilon_0$ and a large number $N$ such that if the initial data
$\psi\in\mathcal{M}(\omega_{KS})$ in  (2.4) satisfies
$\|\psi-\underline{\psi}\|_{C^{2,\alpha}}\leq\frac{\epsilon_0}{N}$,
then there there exist a family of $\sigma_t\in\text{Aut}_r(M)$
such that
\begin{align}\|\varphi_{\sigma_t}-\underline{\varphi_{\sigma_t}}\|_{C^{2,\alpha}}\le
\epsilon_0,~~\forall ~t>0, \end{align}
  where  $\varphi_{\sigma_t}=(\sigma_t)^*\varphi_t+\rho_t$ and $\rho_t$ are
K\"ahler potentials defined by (2.5) in Theorem 2.1.

\end{prop}

\begin{proof}
The proof is a modification of one of Theorem 1.1.
   Let $N_0$ be a
very big number  and choose another big number $N$ with $N_0\le
N\le \frac{1}{\epsilon_0 ^{1/4}}$ such that $C_0e^{-\theta_0
T_N}\le \frac{1}{N_0}$,  where $C_0$ and $T_N$ are two uniform
numbers determined in Corollary 2.7 and Lemma 2.9, respectively.
Now we consider the solution $\varphi=\varphi_{T_N}$ of (2.4) at
time $T_N$.
 By Lemma 5.1 in Appendix 2, we see that there exists  $\sigma=\sigma_{T_N}$ such that for any
$Y\in\eta_r(M)$ with $\int_M\|Y\|^2\omega_{KS}^n=1$, it holds
 \begin{align}
|\int_M \theta_Y' \varphi_\sigma e^{\theta_X}\omega_{KS}^n|\le
O(\epsilon_0^2),
\end{align}
 where $\varphi_\sigma=\sigma^*\varphi+\rho_\sigma$. By adding a constant to  $\varphi_\sigma$
 so that $\widetilde{\varphi_\sigma}=\varphi_\sigma+const.$ satisfies $\int_M
 \widetilde{\varphi_\sigma}e^{\theta_X}\omega_{KS}^n=0$, then we can decompose  $\widetilde {\varphi_\sigma}$
 into $\widetilde{\varphi_\sigma}=\phi+\phi^\bot$ with
$\phi\in\Lambda_1(M,\omega_{KE})$ and
$\phi^\bot\in\Lambda_1^\bot(M,\omega_{KE})$, where
$\Lambda_1^\bot(M,\omega_{KS})$ is a subspace of weighted
$L^2$-integral functions   which are orthogonal to
$\Lambda_1(M,\omega_{KS})\cup\Bbb R$. Then $\phi=\sum_i
a_i\theta_i$ for some constants $a_i$, where $\theta_i$ is a basis
of the space $\Lambda_1(M,\omega_{KS})$.  As a consequence, by
(2.23), we see that $a_i=O(\epsilon_0^2)$ and so
 \beqn \|\phi\|_{C^{2,\alpha}}\le O(\epsilon_0^2).\eeqn

Since $\rho=\rho_\sigma$ satisfies equation,
 $$ \omega_\rho^n=\omega_{KS}^n e^{-\rho-X(\rho)},$$
  by equation (2.4), we have
\begin{align}
\omega_{\widetilde{\varphi_\sigma}}^n=\omega_{KS}^ne^{\sigma^*(\tilde{\dot\varphi})-\widetilde
{\varphi_\sigma}-X(\widetilde{\varphi_\sigma)}+b},
\end{align}
 where $b$ is a constant.
  Then    $\phi^\bot$ is a  solutions of equation
 $$P[\log(\frac{[\omega_{\phi+\phi^\bot}]^n}{\omega_{KS}^n})]+\varphi^\bot+X(\phi^\bot)=P[\sigma^*(\tilde{\dot\varphi})+b-X(\phi)],$$
  where   $P$ is a projection from Banach space
 $H^{2,\alpha}(M)$ to Banach space
  $H^\alpha(M)\cap \Lambda_1^\bot(M,\omega_{KS})$, and $\phi$ and $\sigma^*(\tilde{\dot\varphi})$ are regarded as two peturbation functions.
  Without of the generality, we may assume that
  $$\|\sigma^*(\tilde{\dot\varphi})\|_{C^\alpha}\ge
 C_0\epsilon_0^{\frac{3}{2}},$$
  where $C_0$ is the constant in Corollary 2.7.  Then according to Corollary 2.6 and (2.24), we have
  $$\|P[\sigma^*(\tilde{\dot\varphi})+b-X(\phi)]
 \|_{C^\alpha}=\|P[\sigma^*(\tilde{\dot\varphi})-X(\phi)]\|_{C^\alpha}\le \frac{2\epsilon_0}{NN_0}.$$
  Thus we can use  the Implicity Functional Theorem to get
\begin{align}
\|\phi^\bot\|_{C^{2,\alpha}}\le
2(\|P[\sigma^*(\tilde{\dot\varphi})+b-X(\phi)]
 \|_{C^\alpha} +  \|\phi\|_{C^{2,\alpha}})\le \frac{8\epsilon_0}{NN_0}.
\end{align}
  (2.24) and (2.26) implies
\begin{align}
\|\widetilde{\varphi_\rho}\|_{C^{2,\alpha}}\le
\frac{16\epsilon_0}{NN_0},
\end{align}
so $\widetilde{\varphi_\rho}\in\mathcal{K}(\frac{\epsilon_0}{N})$.

At the next  step (Step 2) we consider equation (2.4) with
$\widetilde{\varphi_\rho}$ as an initial potential to replace
$\psi$. By Lemma 2.9, one sees that equation is solvable for any
$t\in T_N$ with evolved K\"ahler potentials
$\varphi_t^{(2)}\in\mathcal K(\epsilon_0)$ for any $t\le T_N$. So
by the  argument at the last step (Step 1), we can also show that
there exists $\sigma^{(2)}=\sigma_{T_N}^{(2)}\in\text{Aut}_r(M)$
such that
\begin{align}\|\widetilde{\varphi_\rho^{(2)}}\|_{C^{2,\alpha}}
&=\|\widetilde{(\sigma^{(2)})^*\varphi_{T_N}^{(2)}+\rho_{\sigma^{(2)}}}\|_{C^{2,\alpha}}\notag\\
&\le (\frac{16}{N_0})^2 \frac{\epsilon_0}{N}<\frac{\epsilon_0}{N}.
\end{align}
Repeating to use the above step for finite times, we can obtain
$$\|\widetilde{\dot{\varphi^{(k)}}}\|_{C^\alpha}\le
 C_0\epsilon_0^{\frac{3}{2}}$$
  for some integer $k$.
  Then also by using the argument in Step 1,  we can find $\sigma^{(k+1)}=\sigma_{T_N}^{(k+1)}\in \text{Aut}_r(M)$
such that
 $$\|\widetilde{\varphi_\rho^{(k)}}\|_{C^{2,\alpha}}=O(\epsilon_0^{\frac{3}{2}}).$$

Now (Step 3) we  considering  equation (2.4) with
$\widetilde{\varphi_\rho^{(k)}}$ as an initial potential.  Then we
conclude that either  evolved K\"ahler potentials
$\varphi^{(k+1)}_t$ lies in $\mathcal K(\frac{\epsilon_0}{N})$ for
any $t$ or there exists some time $T$ such that
$\|\varphi^{(k+1)}_T\|_{C^{2,\alpha}}=\frac{\epsilon_0}{N}$  for
an evolved K\"ahler potential $\varphi^{(k+1)}_T$ at time $T$. If
the first case happens, then we  will finish all steps. If the
second case happens, then we can repeat the Step1-3 and we can
finally prove that there exist a family of $\sigma_t\in
\text{Aut}_r(M)$ such that (2.22) satisfies for any evolved
K\"ahler potential $\varphi_t$ of  (2.4) at $t$ as long as the
initial potential $\psi$ lies in
$\mathcal{K}(\frac{\epsilon_0}{N})$.

\end{proof}

\begin{proof}[Proof of Theorem 2.1] We suffice to do higher-order
estimates for the modified evolved K\"ahler potentials
$((\sigma_t)^*\varphi_t+\rho_t)$ of equation (2.4) in Proposition
2.10. Here we use a trick in [CT2] to choose a modified family of
holomorphic transformations $\overline \sigma_t\in\text{Aut}_r(M)$
($0< t<\infty$, $\overline\sigma_0=\text{Id}$) to replace
$\sigma_t$ such that for any $t\in (0,\infty)$ (cf. [TZ3]),
$$\|\sigma_t^{-1}\overline \sigma_t-Id\|\le C,$$
and
$$\|(\overline \sigma_t^{-1})_*\frac
{\partial\overline \sigma_t}{\partial t}\|_g\le C,$$
 where $(\overline \sigma_t^{-1})_*\frac {\partial\overline
\sigma_t}{\partial t}=\overline X_t\in \eta_r(M)$ is  a family of
holomorphic vector fields on $M$. Furthermore,  for any $k \ge 0$,
we may assume that there is a constant $C_k$ such that
$$\|\frac{\partial^k \overline X_t}{\partial t^k}\|_g\le C_k.$$
Note that the choice of  such $\overline\sigma_t$ just depends on
the $C^0$-estimate of
$\widetilde{\varphi_\sigma}=\widetilde{((\sigma_t)^*\varphi_t+\rho_t)}$.
Thus by Proposition 2.10, we also have
$(\overline\varphi-\frac{1}{V}\int_M\overline\varphi\omega_{KS}^n)\in\mathcal
K(\epsilon_0)$. On the other hand, by equation (2.4), the new
modified potential
$\overline\varphi=\varphi_{\overline\sigma_t}=(\overline\sigma_t)^*(\varphi_t+\overline\rho_t)$
will satisfy equation,
 \begin{align} &\frac{\partial\overline\varphi}{\partial
t}=\log\frac{\omega^n_{\overline\varphi}}{\omega_{KS}^n}+\overline\varphi+\overline X(\overline\varphi), \notag\\
&\overline\varphi(0)=\psi-\underline{\psi}.
\end{align}
 Now for each $t$, we can consider  solution $\varphi'$ of equation (2.29) on the interval $[t-1, t+1]$ with
 $(\overline\varphi_{t-1}-\frac{1}{V}\int_M\overline\varphi_{t-1}\omega_{KS}^n)$
 as an initial data. By the Maximal Principle, it is easy to see that both
 $\varphi_s'$ and $\dot\varphi_s'$ are uniformly bounded in $[t-1, t+1]$. Since
 $$\|\overline\varphi_s'-\frac{1}{V}\int_M\overline\varphi_s'\omega_{KS}^n\|_{C^{2,\alpha}}=
 \|\overline\varphi_s-\frac{1}{V}\int_M\overline\varphi_s\omega_{KS}^n\|_{C^{2,\alpha}},$$
 by   the regularity theory of parabolic equation,  we  get all
  bounded $C^k$-estimates for $\overline\varphi_t'$. This implies
that all $C^k$-norms of
$(\overline\varphi_t-\frac{1}{V}\int_M\overline\varphi_t\omega_{KS}^n)$
are uniformly bounded, and so are $\widetilde{\varphi_\sigma}$.

From the above estimates, we see that for any sequence of K\"ahler
metrics $\omega_{\varphi_{\sigma_i}}$, there exists a limit
K\"ahler metric $\omega_{\infty}$ of subsequence of
$\omega_{\varphi_{\sigma_i}}$ in the sense of $C^k$-convergence.
By applying Perelman's $W$-function in [Pe] to the normalized
Ricci equation (0.1), one conculdes that $\omega_{\infty}$ must be
a K\"ahler-Ricci soliton (cf. [Se]). Since the K\"ahler-Ricci
solition is  unique,  we see that there exists an element
$\tau_\infty\in \text{Aut}_0(M)$ such that
$\omega_{\infty}=\tau_\infty^*\omega_{KS}$.  By using the fact
that the convergent sequence is  arbitary,  the above implies that
there exists a family of $\tau=\tau_t\in\text{Aut}_0(M)$ such that
evolved K\"ahler metrics $\tau^*g$ converge to $g_{KS}$ smoothly.

If in addition that the initial K\"ahler potential  $\psi$ is
$K_{X_0}$-invariant, by Remark 2.7,  we can follow  the argument
in the proof of Theorem 1.1 to apply the Implicity Functional
Theorem to equation (2.25) in  the proof of Proposition 2.10  to
show that there exists a family of
$\sigma=\sigma_t\in\text{Aut}_r(M)$ such that the modified
solution $\varphi_\sigma=((\sigma_t)^*\varphi_t+\rho_t)$ of
equation (2.25) satisfy
\begin{align}
\|\widetilde{\varphi_\sigma}\|_{C^{2,\alpha}}
&=\|((\sigma_t)^*\varphi_t+\rho_t-\frac{1}{V}\int_M(\sigma_t)^*\varphi_t+\rho_t))\omega_{KS}^n)\|_{C^{2,\alpha}}\notag\\
&\le 2\|P(\sigma^*(\tilde{\dot\varphi})\|_{C^\alpha}\le
Ce^{-\theta't},~~\forall~t>0.\end{align}
 Similarly, we can also prove that for any $k$ it holds
$$ \|\widetilde{\varphi_\sigma}\|_{C^{k+2,\alpha}}\le
C_ke^{-\theta't},~~\forall~t>0,$$
 since by Remark 2.8 and the embedding theory of Sobolev spaces we
 have
$$ \|\tilde{\dot\varphi}\|_{C^{k,\alpha}}\le
C_k'e^{-\theta't},~~\forall~t>0,$$
 where $C_k$ and $C_k'$ are uniform constants which  depends
 only on $k, \epsilon_0$ and higher-order derivatives of the initial
 K\"ahler potential $\psi$. Therefore we prove that K\"ahler metrics
$\sigma^*(\omega_{\varphi})$ converge exponentially to the
K\"ahler-Ricci soliton  $\omega_{KS}$.
\end{proof}

\section {Uniqueness of  the limit of  K\"ahler Ricci flow}

By Theorem 1.1 and Theorem 2.2 in Section 1 and 2, we complete the
proof of Theorem 0.1. As an application of Theorem 0.1, we have
the following uniqueness result about the limit of K\"ahler-Ricci
flow.

\begin{theo} Let $g_t$ be the evolved K\"ahler metrics of
K\"ahler-Ricci flow (0.1) on $M$. Suppose that there exists a
sequence $g_i$ of $g_t$ and a sequence of holomorphic
transformations $\sigma_{i}\in\text{Aut}(M)$ such that
$\sigma_{i}^*g_i$ converge to a limit K\"ahler metric $g_\infty$
in the sense of $C^{2,\alpha}$-norm of K\"ahler potentials. Then
the K\"ahler-Ricci flow converges to  $g_\infty$  smoothly in the
sense of Cheeger-Gromov.
\end{theo}

\begin{proof} First we note that by applying Perelman's $W$-function in [Pe],
the limit K\"ahler metric $g_\infty$ must be a K\"ahler-Ricci
soliton $g_{KS}$ on $M$. On the other hand, by the convergence of
$g_i$, one sees that for any $\epsilon<<1$ there exists a big
index $i$ such that the potential $\psi=\psi_i$ of $g_i$ satisfies
$$\|\psi-\underline{\psi}\|_{C^{2,\alpha}}\le \epsilon,$$
where $\omega_{\psi}=\omega_{KS}+\sqrt{-1}\p\bar\p\psi$. Now we
consider the K\"ahler-Ricci flow (0.1) with $\omega_g=\omega_\psi$
as an initial K\"ahler metric. Then by Theorem 0.1, this flow
converges to $g_{KS}$ smoothly in the sense of Cheeger-Gromov, so
the theorem is proved.

\end{proof}

\begin{rem} In a subsequence paper, we will prove the uniqueness of the limit of
K\"ahler-Ricci flow in more general.  Namely, Theorem 3.1 is still
true if we assume that there exists a sequence $g_i$ of $g_t$ of
equation (0.1) which  converge to a limit Riemanian metric
$g_\infty$ in $C^{2,\alpha}$-norm in the sense of
 Cheeger-Gromov.
\end{rem}

\section{Appendix 1}

In this appendix, we prove a lemma  about $W^{k,2}$-estimates of
$\dot\varphi$ for evolved K\"ahler metrics $\varphi$ of  flow
(1.1) under the assumption $\varphi\in\mathcal{K}(\epsilon_0)$.
Recall that a $k-norm$ $\|\nabla^k{\dot\varphi}\|^2$ is  defined
by
$$\|\nabla^k\dot\varphi\|^2=\sum
g^{i_1j_1}...g^{i_kj_k}{\dot\varphi}_{i_1...i_k}{\dot\varphi}_{j_1...j_k},$$
 where ${\dot\varphi}_{i_1...i_k}$ are components of  the $k$-covariant
 derivative of $\dot\varphi$ with respect  to
 $g=\omega_\varphi$ as a Riemannian metric.

Since
$${\dot\varphi}_{i_1...i_k}=\frac{\partial^k\dot\varphi}{\partial
x^{i_1}...\partial
x^{i_k}}+\Phi_1(\dot\varphi,...,{\dot\varphi}_{i_1...i_{k-1}}),$$
 we have
\begin{align}
\frac{ d
{\dot\varphi}_{i_1...i_k}}{dt}&=\frac{\partial^k\ddot{\varphi}}{\partial
x^{i_1}...\partial x^{i_k}}+\frac{d\Phi_1}{dt}\notag\\
&={\ddot{\varphi}}_{i_1...i_k}+\Phi_2(\dot\varphi_i,...,{\dot\varphi}_{i_1...i_{k-1}})+\frac{d\Phi_1}{dt}\notag,
\end{align}
where $\Phi_1$ and $\Phi_2$ are two polynomials with variables
$\dot\varphi_i,...,{\dot\varphi}_{i_1...i_{k-1}}$ and coefficients
$g_{ij}, \partial^l g_{ij}, l=1,...,k$.  Note that
$\frac{d\Phi_1}{dt}$ is uniformly bounded.  Then by equations
(0.1) and (1.1), one can estimate
 \begin{align}
&\frac{d\|\nabla^k \dot\varphi\|^2}{dt}\notag\\
&=\sum_{i_1,...,i_k} \sum_\alpha (R_{ i_\alpha
i_\alpha}-g_{i_\alpha
i_\alpha}){\dot\varphi}_{i_1,...,i_\alpha,...,i_k}{\dot\varphi}_{i_1,...,i_\alpha,...,i_k}\notag\\
&+ 2\sum
g^{i_1j_1}...g^{i_kj_k}\frac{d{\dot\varphi}_{i_1...i_k}}{dt}{\dot\varphi}_{j_1...j_k}\notag\\
 &\le C_1\|\nabla^k\dot\varphi\|^2+C_2\|\nabla\dot\varphi\|^2+
2(\ddot{\varphi})_{i_1...i_k}{\dot\varphi}_{j_1...j_k}\notag\\
& \le
-2\|\nabla^{k+1}\dot\varphi\|^2+C_1'\|\nabla^k\dot\varphi\|^2+C_2'\|\dot\varphi-c(t)\|^2.
\end{align}

 Let
 $$H_k(t)=\int_M\|\nabla^k\dot\varphi\|^2\omega_{\varphi}^n.$$
Then by (4.1),  we have

\begin{lem} Let $T$ be any positive number.  Suppose that $\varphi_t$ lies
$\mathcal{K}(\epsilon_0)$  for any $t\in [0,T)$. Then
 \begin{align} H_k(t)\le
Ce^{-\theta' t}, ~~\forall ~t\in [0,T).\end{align}
\end{lem}

\begin{proof} By (4.1), we have
\begin{align}
\frac{ dH_k(t)}{dt}&=\int_M
\frac{d\|\nabla^k\dot\varphi\|^2}{dt}\omega_{\varphi}^n+\int_M
\|\nabla^k\dot\varphi\|^2\triangle\dot\varphi \omega_{\varphi}^n\notag \\
&\le -2H_{k+1}(t)+C_3H_k(t)\notag+C_2'\|\dot\varphi-c(t)\|^2\\
&\le -\theta' H_{k}(t)+C_4H_0(t).
\end{align}
On the other hand, from the proof of Lemma 1.3, we in fact prove
that
$$
\frac{d H_0(t)}{dt}\le -\theta H_0(t), ~~\forall ~t\in [0,T),$$
 if $\varphi\in \mathcal{K}(\epsilon_0),~~\forall ~t\in [0,T)$.
Thus  Combining the above  inequality with (4.3), we get
$$\frac{ d(H_k(t)+AH_0(t))}{dt}\le -\theta' [H_{k}(t)+\frac{(A\theta-C_4)}{\theta'}H_0(t)],$$
  where $A$ is a sufficiently  large number.  It follows
$$\frac{ d \ln(H_k(t)+AH_0(t))}{dt}\le -\theta'\frac{ H_{k}(t)+\frac{(A\theta-C_4)}{\theta'}H_0(t)}{H_k(t)+AH_0(t)}\le -\theta'.$$
Thus
 $$ H_k(t)+AH_0(t)\le (H_k(0)+AH_0)e^{-\theta' t}$$
  and so (4.2) follows.
\end{proof}

\section{Appendix 2}

The following lemma  is about  the existence of almost
orthonormality of a K\"ahler potential to the space of first
eigenvalue-functions of operator $(P,\omega_{KS})$ defined in
Lemma 2.2 in Section 2. The  lemma  is crucial  in the proof of
Proposition 2.10.

\begin{lem}    Let $M$ be  a compact K\"ahler manifold $M$ with
$c_1(M)>0$ which admits a K\"ahler-Ricci soliton $(\omega_{KS},
X_0)$. Then for any  K\"ahler potential
$\phi\in\mathcal{K}(\epsilon_0)$ there exists a $\sigma\in
\text{Aut}_r(M)$ with bounded $\text{dist}(\sigma,Id)$  such that
for any $Y\in\eta_r(M)$ with $\int_M\|Y\|^2\omega_{KS}^n=1$, it
holds
$$ |\int_M \theta_Y'
(\sigma^*\phi+\rho_\sigma)e^{\theta_X}\omega_{KS}^n|\le
C\|X'(\phi)\|_{C^0}^2=O(\epsilon_0^2),$$
 where $\theta_Y\in\text{ker}(P,\omega_{KS})$ and
 $\rho_\sigma$ is a K\"ahler potential  defined by (2.5) in Section 2.
 \end{lem}

\begin{proof}This lemma was proved in [TZ1] if $\phi$ is
$K_0$-invariant. The key point in the proof  is to use a
functional defined on a space of K\"ahler-Ricci solitons
 $$\{\omega_{KS}'=\sigma^*(\omega_{KE})=\omega_{KS}+
 \sqrt{-1}\p\bar\p\rho_\sigma|~~\sigma\in\text{Aut}_r(M)\},$$
  which  was introduced  in [Zh] by
 \beqs
&&(I-J)(\omega_\phi,\omega_{KS}')\\
&&=\int_0^1 dt\int_M
\dot\phi_te^{\theta_{X_0}(\phi_t)}\omega_{\phi_t}^n -\int_M
(-\phi+\rho)e^{\theta_{X_0}+X(\rho)}(\omega_{KS}')^n,\eeqs
 where $\phi_t$ is a $K_{X_0}$-invariant path in
 $\mathcal{M}(\omega_{KS})$ which connects $0$ and $-\phi+\rho$,  and
 $\theta_{X_0}(\phi_t)$ are  potentials  of $X_0$ associated to
 metric $\omega_{\phi_t}$ defined by (2.1). It is proved in [Zh]
 that this well-defined for a $K_0$-invariant $\phi$, i.e.,  the functional is independent of
 the choice of a $K_0$-invariant path. But for a general K\"ahler potential $\phi$, one can also
 show that $(I-J)(\omega_\phi,\omega_{KS}')$ is not well-defined (to see (5.4) below),
 so  we shall introduce another functional defined on whole space $\mathcal{M}(\omega_{KS})$
 to replace it. In fact, we consider the following  functional
 \begin{align}
\mathcal{F}(\omega_\phi,\omega_{KS}')&=\text{Re}[\int_0^1 dt\int_M
(-\phi+\rho_\sigma)e^{\theta_{X_0}(t(-\phi+\rho_\sigma))}\omega_{t(-\phi+\rho_\sigma)}^n\notag\\
&-\int_M
(-\phi+\rho)e^{\theta_{X_0}'}(\omega_{KS}')^n].\end{align}
 Clearly, the definition of $\mathcal{F}$ just uses a real part of
 $(I-J)(\omega_\phi,\omega_{KS}')$ while a K\"ahler potentials path is chosen by
 $\phi_t=t(-\phi+\rho_\sigma)$. We now consider a  K\"ahler potentials path $\rho_t$
 induced by an one-parameter subgroup $\sigma_t$ generated by the
 real part of $Y\in \eta_r(M)$, i.e. $\rho_t$ are defined by
 $\omega_t=\sigma_t^*\omega_{KS}'=\omega_{KS}'+\sqrt{-1}\p\bar\p\rho_t$.
Let
 \begin{align} f_Y(t)=\text{Re}[\int_0^{1+t} ds\int_M
(\dot\phi_s)e^{\theta_{X_0}(\phi_s)}\omega_{\phi_s}^n -\int_M
(-\phi+\rho_t)e^{\theta_{X_0}(\omega_t)}\omega_t^n],
\end{align}
  where $\phi_s$ is a path in $\mathcal{M}(\omega_{KS})$ defined by
$\phi_s=s(-\phi+\rho_\sigma), ~\forall ~0\le s\le 1$ and
$\phi_s=-\phi+\rho_\sigma+\rho_t$, $1\le s\le 1+t$.  It is easy to
see
$$
 \frac{d}{dt} f_Y(t)|_{t=0}=\int_M \theta_Y'
(-\varphi+\rho_\sigma)e^{\theta_X'}(\omega_{KS}')^n.
$$
This implies
 \begin{align}
 \frac{d}{dt} f_Y(t)|_{t=0}=-\int_M \theta_Y'
((\sigma^{-1})^*\varphi+\rho_{\sigma^{-1}})e^{\theta_X}\omega_{KS}^n
\end{align}

 The gap between $ f_Y(t)$ and
$\mathcal{F}(\omega_\phi,\omega_t)$ can be computed as follows.
Let $\Delta=\{(\tau,s)|~0\le\tau\le 1,~0\le s\le
\tau+(1-\tau)(1+t)$\} be a domain in $\Bbb R^2$. Let
$\Phi=\Phi(\tau,s;\cdot)$ be K\"ahler potentials with two
parameters $(\tau, s)\in\Delta$ which satisfy:
 \begin{align}
 &\Phi=s(-\phi+\rho_\sigma+\rho_t),~~0\le s\le 1,~~~\text{as}~ \tau=1;\notag\\
 &\Phi=\phi_s, ~0\le s\le 1+t,~\text{as}~ \tau=0;\notag\\
&\Phi=0,~\text{as}~
s=0;\Phi=-\phi+\rho_\sigma+\rho_t,~s=\tau+(1-\tau)(1+t). \notag
\end{align}
 Then by using the Stoke's formula, we have
 \begin{align}&
|f_Y(t)-\mathcal{F}(\omega_\phi,\omega_t)|\notag\\
 &=|\text{Re}\{\int_{\partial \Delta} \int_M
d_{\tau,s}\Phi(\tau,s;\cdot)
e^{\theta_{X_0}(\phi_s)}\omega_{\phi_s}^n \}|\notag\\
&=|\text{Re}\{\int_\Delta d\tau ds\int_M \dot\Phi_\tau
(<\overline\partial\dot\Phi_s,\overline\partial\theta_{X_0}(\Phi)>-\notag\\
&<\overline\partial\theta_{X_0}(\Phi),\overline\partial\dot\Phi_s>)e^{\theta_{X_0}(\Phi)}\omega_{\Phi}^n\}|\notag\\
&=2|\text{Re}\{\int_\Delta d\tau ds\int_M \dot\Phi_\tau
\text{Im}(X_0(\Phi_s))e^{\theta_{X_0}(\Phi)}\omega_{\Phi}^n\}|\notag\\
&\le C\|X'(\phi)\|_{C^0}^2.
\end{align}
At the last inequality, we used a fact that  $X_0(\rho_\sigma)$
and $X_0(\rho_t)$ are both real-valued.  Similarly,  we can get
 \begin{align}|\frac{d}{dt}(f_Y(t)-\mathcal{F}(\omega_\phi,\omega_t))|_{t=0}
\le
 C\|X'(\phi)\|_{C^0}^2.\end{align}

 Next  we claim
 \begin{align}
 \mathcal{F}(\sigma)=\mathcal{F}(\omega_\phi,\omega_{KS}')\ge 0.\end{align}
 To prove the claim, we let
 \begin{align}g(t)&=\text{Re}[\int_0^t ds\int_M
(-\phi+\rho_\sigma)e^{\theta_{X_0}(s(-\phi+\rho_\sigma))}\omega_{s(-\phi+\rho_\sigma)}^n\notag\\
&-\int_M
(-\phi+\rho)e^{\theta_{X_0}(t(-\phi+\rho_\sigma))}\omega_{t(-\phi+\rho_\sigma)}^n].\notag
\end{align}
Then
$$\mathcal{F}(\sigma)=g(1)=\int_0^1 g(t)'dt.$$
On the other hand,  we have
\begin{align}  g(t)'&=\text{Re}[n\sqrt{-1}\int_M
 \partial(-\phi+\rho_\sigma)\wedge\overline\partial(-\phi+\rho_\sigma)
 e^{\theta_{X_0}(t(-\phi+\rho_\sigma))}\omega_{t(-\phi+\rho_\sigma)}^n]\notag\\
&\ge 0.\notag
\end{align}
Thus we get $g(1)\ge 0$ and prove the claim.

 By the above claim, we can take a
minimizing sequence of $\mathcal{F}(\sigma)$ in $\text{Aut}_r(M)$
and  we see that  for any small $\epsilon\le \epsilon_0$, there
exists a $\sigma\in\text{Aut}_r(M)$ with bounded
$\text{dist}(\sigma,Id)$ such that for any $Y\in\eta_r(M)$ with
$\int_M\|Y\|^2\omega_{KS}^n=1$, we have
 \begin{align} |D\mathcal{F}(\sigma)(Y)|\le \epsilon.\end{align}
 Therefore
combining (5.3), (5.5) and (5.7), we prove the lemma while
$\sigma$ is replaced by $\sigma^{-1}$.

\end{proof}

\end{document}